\documentclass[11pt,a4paper]{article}
\usepackage{amsmath}
\usepackage{amsfonts}
\usepackage{amssymb}
\usepackage{graphicx}
\usepackage{tabularx}
\usepackage{booktabs}
\usepackage{lscape}
\usepackage[ruled]{algorithm2e}

\newtheorem{thm}{Theorem}[section]

\newtheorem{cor}[thm]{Corollary}

\newtheorem{de}{Definition}[section]

\newenvironment{proof} {\par \noindent \textbf{Proof. }}{\QED \par \bigskip \par}
\newcommand{\QED}{\hfill$\square$}

\def\cp{\,\square\,}

\begin{document}
\baselineskip 16pt

\phantom{.} \vskip 6cm

\begin{center}
{\LARGE \bf ECCENTRIC CONNECTIVITY INDEX } \\
\bigskip
\bigskip

{\large \sc Aleksandar Ili\'c\ \footnotemark[3] }

\smallskip
{\em Faculty of Sciences and Mathematics, Vi\v{s}egradska 33, 18 000 Ni\v{s}, Serbia} \\
e-mail: {\tt aleksandari@gmail.com}

\bigskip\medskip
{\small (Received \today)}
\bigskip

\end{center}


\begin{abstract}
The eccentric connectivity index $\xi^c$ is a novel distance--based molecular structure descriptor
that was recently used for mathematical modeling of biological activities of diverse nature. It is
defined as $\xi^c (G) = \sum_{v \in V (G)} deg (v) \cdot \varepsilon (v)$\,, where $deg (v)$ and
$\varepsilon (v)$ denote the vertex degree and eccentricity of $v$\,, respectively. We survey some
mathematical properties of this index and furthermore support the use of eccentric connectivity
index as topological structure descriptor. We present the extremal trees and unicyclic graphs with
maximum and minimum eccentric connectivity index subject to the certain graph constraints. Sharp
lower and asymptotic upper bound for all graphs are given and various connections with other
important graph invariants are established. In addition, we present explicit formulae for the
values of eccentric connectivity index for several families of composite graphs and designed a
linear algorithm for calculating the eccentric connectivity index of trees. Some open problems and
related indices for further study are also listed.
\end{abstract}

\section{Introduction}

Let $G$ be a simple connected graph with $n = |V|$ vertices. For a vertex $v \in V (G)$\,, $deg
(v)$ denotes the degree of $v$\,. For vertices $v, u \in V$\,, the distance $d (v, u)$ is defined
as the length of the shortest path between $v$ and $u$ in $G$\,. The eccentricity $\varepsilon (v)$
of a vertex $v$ is the maximum distance from $v$ to any other vertex. The diameter $d (G)$ of $G$
is defined as the maximum value of the eccentricities of the vertices of $G$\,. Similarly, the
radius of $G$ is defined as the minimum value of the eccentricities of the vertices of $G$\,.

Sharma, Goswami and Madan \cite{ShGoMa97} introduced a distance--based molecular structure
descriptor, which they named ''{\it eccentric connectivity index\/}'' and which is defined as
$$
\xi^c = \xi^c (G) = \sum_{v \in V (G)} deg (v) \cdot \varepsilon (v) \ .
$$

The index $\xi^c$ was successfully used for mathematical models of biological activities of diverse
nature \cite{DuGuMa08,GuSiMa02,KuSaMa04,SaMa00}. The prediction of physico-chemical,
pharmacological and toxicological properties of compounds directly from their molecular structure
has become an important tool in pharmaceutical drug design. The parameters derived from a
graph-theoretic model of a chemical structure are being used not only in QSAR studies, but also in
the environmental hazard assessment of chemicals. Eccentric connectivity index has been shown to
give a high degree of predictability of pharmaceutical properties, and provide leads for the
development of safe and potent anti-HIV compounds
\cite{DuMa06,GuSiMa02a,KuMa06,KuMa07,LaMa05,SaMa03a,SaMa02}. Eccentric connectivity index is also
proposed as a measure of branching in alkanes \cite{IlGu10}.

The investigation of its mathematical properties started only recently, and has so far resulted in
determining the extremal values of the invariant and the extremal graphs where those values are
achieved \cite{ZhDu09} and \cite{IlGu10}, and also in a number of explicit formulae for the
eccentric connectivity index of several classes of graphs \cite{DoSa09}. An asymptotically sharp
upper bound is derived independently in \cite{DoSaVu10} and \cite{MoMuSw09}, and it is shown that
the eccentric connectivity index grows no faster than a cubic polynomial in the number of vertices.
In \cite{Il10} the present author analyzed the eccentric connectivity index of unicyclic graphs
with given girth and designed a linear algorithm for calculating $\xi^c$ of unicyclic graphs. In
\cite{AsDoSa10} and \cite{AsSaGh10} Ashrafi et al. computed the eccentric connectivity index of
nanotubes and nanotori. In \cite{DoSa09a} the authors presented explicit formulae for the values of
eccentric connectivity index and total eccentricity for several families of benzenoid graphs
(triangular benzenoids, benzenoid rhombi, parallelograms and hexagonal benzenoids), while in
\cite{DoGrOr10} are given formulae of zigzag and armchair hexagonal belts and the corresponding
open chains.

The paper is organized as follows. In Section 2 we present some preliminary results and explicit
formulae for several important classes of graphs. In Section 3, we introduce two graph
transformations, such that $\xi^c$ is monotone under these transformations. In Section 4, some
upper and lower bounds involving other important graph invariants are established. In addition, we
prove that the star $S_n$ has minimal value of eccentric connectivity index, while lolipop $LP_{n,
\lfloor n / 3 \rfloor}$ has asymptotically the largest value of eccentric connectivity index among
connected graphs. In Section 5, we present the extremal trees with maximum and minimum eccentric
connectivity index subject to certain constraints (diameter, radius, the number of pendent
vertices, matching number, independence number, maximum vertex degree). In Section 6 we
characterize the extremal unicyclic graphs with $n$ vertices and given girth $k$\,. In Section 7,
the formulae for eccentric connectivity index of Cartesian product of graphs are given, with some
chemical examples including nanotorus and nanotube. In Section 8 we designed a linear algorithm for
calculating $\xi^c$ index of trees. Some open problems and modifications of the eccentric
connectivity index are presented in Section 9.

\section{Preliminaries}

First we give two useful estimations of the eccentric connectivity index and in the following we
present explicit formulae for eccentric connectivity index of various families of graphs.

For every vertex $v \in G$\,, there holds $r (G) \leq \varepsilon (v) \leq d (G)$\,. Using $\sum_{v
\in V (G)} deg (v) = 2m$\,, it follows
\begin{equation}
2m \cdot r (G) \leq \xi^c (G) \leq 2m \cdot d(G)\,,
\end{equation}
with equality if and only if all vertices have the same eccentricity, i.e. $r (G) = d (G)$\,.

Let $k$ be the number of vertices of degree $n - 1$ in the graph $G \not \cong K_n$\,, with $r (G)
= 1$\,. For $1 \leq k \leq n - 1$ we have
\begin{equation}
\label{eq:k} \xi(G) = \sum_{v \in V (G)} deg (v) \cdot \varepsilon (v) \geq (n - 1)k + 2 (2m - (n -
1)k) = 4m - k (n - 1)\,,
\end{equation}
with equality if and only if all $n - k$ vertices of degree less than $n - 1$ have eccentricity
two.

By direct calculation, the following formulae hold
$$
\xi^c (K_n) = n (n - 1) \qquad \xi^c (K_{n,m}) = 4mn
$$
$$
\xi^c (P_n) = \left \lfloor \frac{3 (n - 1)^2 + 1}{2} \right \rfloor \qquad \xi^c (S_n) = 3 (n - 1)
$$
$$
\xi^c (C_n) = 2n \left \lfloor \frac{n}{2} \right \rfloor \qquad \xi^c (Q_n) = n^2 \cdot 2^n
$$
for the complete graph $K_n$\,, complete bipartite graph $K_{n,m}$\,, path $P_n$\,, star $S_n$\,,
cycle $C_n$ and hypercube~$Q_n$\,.

It is sometimes interesting to consider the sum of eccentricities of all vertices of a given
graph~$G$~\cite{DaGoSw04}. We call this quantity the total eccentricity of the graph $G$ and denote
it by $\zeta (G)$\,. For a $k$-regular graph $G$\,, we have $\xi^c (G) = k \cdot \zeta (G)$\,.

\section{Graph transformations}

\begin{thm}
\label{thm-pi} Let $w$ be a vertex of a nontrivial connected graph $G$\,. For nonnegative integers
$p$ and $q$\,, let $G (p, q)$ denote the graph obtained from $G$ by attaching to vertex $w$ pendent
paths $P = w v_1 v_2 \ldots v_p$ and $Q = w u_1 u_2 \dots u_q$ of lengths $p$ and~$q$\,,
respectively. If $p \geq q \geq 1$\,, then
$$
\label{eq-pi} \xi^c (G (p, q)) < \xi^c (G (p + 1, q - 1)) \ .
$$
\end{thm}

\begin{proof}
The degrees of vertices $u_{q - 1}$ and $v_p$ are changed, while all other vertices have the same
degree in $G (p + 1, q - 1)$ as in $G (p, q)$\,. Since after this transformation the longer path
has increased, the eccentricity of vertices from $G$ are either the same or increased by one. We
will consider three cases based on the longest path from the vertex $w$ in the graph $G$\,. Denote
with $deg' (v)$ and $\varepsilon' (v)$ the vertex degree and eccentricity of vertex $v$ in $G (p +
1, q - 1)$\,.

\noindent {\bf Case 1. } The length of the longest path from the vertex $w$ in $G$ is greater than
$p$\,. This means that the vertex of $G$\,, most distant from $w$ is the most distant vertex for
all vertices of $P$ and $Q$\,. It follows that $\varepsilon_{G (p + 1, q - 1)} (v) = \varepsilon_{G
(p, q)} (v)$ for all vertices $w, v_1, v_2, \ldots, v_p, u_1, u_2, \ldots, u_{q - 1}$\,, while the
eccentricity of $u_q$ increased by $p + 1 - q$\,.
\begin{eqnarray*}
\xi^c (G (p + 1, q - 1)) - \xi^c (G (p, q)) &\geq& \left[ deg' (u_{q - 1})\,\varepsilon' (u_{q -
1})
+ deg' (u_{q})\,\varepsilon'\,(u_{q}) + deg' (v_{p})\,\varepsilon' (v_{p}) \right] \\
&-& \left[ deg (u_{q - 1})\,\varepsilon (u_{q - 1}) + deg (u_{q})\,\varepsilon (u_{q}) +
deg (v_{p})\,\varepsilon (v_{p}) \right] \\
&=& - \varepsilon (u_{q - 1}) + (p - q + 1) + \varepsilon (v_p) > 0\,.
\end{eqnarray*}

\noindent {\bf Case 2. } The length of the longest path from the vertex $w$ in $G$ is less than or
equal to $p$ and greater than $q$\,. This means that either the vertex of $G$ that is most distant
from $w$ or the vertex $v_p$ is the most distant vertex for all vertices of $P$\,, while for
vertices $w, u_1, u_2, \ldots, u_q$ the most distant vertex is $v_p$\,. It follows that
$\varepsilon_{G (p + 1, q - 1)} (v) = \varepsilon_{G (p, q)} (v)$ for vertices $v_1, v_2, \ldots,
v_p$\,, while $\varepsilon_{G (p + 1, q - 1)} (v) = \varepsilon_{G (p, q)} (v) + 1$ for vertices
$w, u_1, u_2, \ldots, u_{q - 1}$\,. The eccentricity of $u_q$ increased by at least $1$\,.
\begin{eqnarray*}
\xi^c (G (p + 1, q - 1)) - \xi^c (G (p, q)) &\geq& deg' (w)\,\varepsilon' (w) +
deg' (v_{p})\,\varepsilon' (v_{p}) + \sum_{j = 1}^q deg' (u_{j})\,\varepsilon' (u_{j})\\
&-& deg (w)\,\varepsilon (w) - deg (v_{p}) \,\varepsilon (v_{p}) -
\sum_{j = 1}^q deg (u_{j})\,\varepsilon (u_{j})\\
&\geq& q + \left[ \varepsilon (u_{q - 1}) + 1 \right] \left[ deg (u_{q - 1}) - 1 \right] -
\varepsilon (u_{q - 1})\,deg (u_{q - 1}) + \varepsilon (v_p)\\
&>& \varepsilon (v_p) - \varepsilon (u_{q - 1}) > 0 \ .
\end{eqnarray*}

\noindent {\bf Case 3. } The length of the longest path from the vertex $w$ in $G$ is less than or
equal to $q$\,. This means that the pendent vertex most distant from the vertices of $P$ and $Q$ is
either $v_p$ or $u_q$\,, depending on the position. Using the formula for eccentric connectivity
index of a path, we have
\begin{eqnarray*}
\xi^c (G (p + 1, q - 1)) - \xi^c (G (p, q)) &>& \xi^c (P_{p + q + 1}) +
[deg (w) - 2]\,\varepsilon' (w) \\
&-& \xi^c (P_{p + q + 1}) - [deg (w) - 2]\,\varepsilon (w) \\
&=& deg (w) - 2 \geq 0 \ .
\end{eqnarray*}
Since $G$ is a nontrivial graph with at least one vertex, we have strict inequality.

This completes the proof.
\end{proof}

\begin{de}
Let $v$ be a vertex of a tree $T$ of degree $m + 1$\,. Suppose that $P_1, P_2, \ldots, P_m$ are
pendent paths incident with $v$\,, with lengths $1 \leq n_1 \leq n_2 \leq \ldots \leq n_m$\,. Let
$w$ be the neighbor of $v$ distinct from the starting vertices of paths $v_1, v_2, \ldots, v_m$\,,
respectively. We form a tree $T' = \delta (T, v)$ by removing the edges $v v_1, v v_2, \ldots, v
v_{m - 1}$ from $T$ and adding $m - 1$ new edges $w v_1, w v_2, \ldots, w v_{m - 1}$ incident with
$w$\,. We say that $T'$ is a $\delta$-transform of $T$ and write $T' = \delta (T, v)$\,.
\end{de}

\begin{thm}
\label{thm-delta} Let $T' = \delta (T, v)$ be a $\delta$-transform of a tree $T$ of order $n$\,.
Let $v$ be a non-central vertex, that is furthest from the root among all branching vertices (with
degree greater than~$2$). Then
$$
\xi^c (T) > \xi^c (T')\,.
$$
\end{thm}

\begin{proof}
The degrees of vertices $v$ and $w$ have changed -- namely, $deg (v) - deg' (v) = deg' (w) - deg
(w) = m - 1$\,. Since the furthest vertex from $v$ does not belong to $P_1, P_2, \ldots, P_m$ and
$n_m \geq n_i$ for $i = 1, 2, \ldots, m - 1$\,, it follows that the eccentricities of all vertices
different from $P_1, P_2, \ldots, P_{m - 1}, P_m$ do not change after $\delta$ transformation. The
eccentricities of vertices from $P_m$ also remain the same, while the eccentricities of vertices
from $P_1, P_2, \ldots, P_{m - 1}$ decrease by one. Using the equality $\varepsilon (v) =
\varepsilon (w) + 1$\,, it follows that
\begin{eqnarray*}
\xi^c (T) - \xi^c (T') &=& \sum_{i = 1}^{m - 1} (1 + 2(n_i - 1)) + (m - 1) \cdot \varepsilon (v) - (m - 1) \cdot \varepsilon (w) \\
&=& 2 \left ( n_1 + n_2 + \ldots + n_{m - 1} \right) - (m - 1) + (m - 1)(\varepsilon (v) - \varepsilon (w)) \\
&=& 2 \left ( n_1 + n_2 + \ldots + n_{m - 1} \right) > 0\,.
\end{eqnarray*}
This completes the proof.
\end{proof}

By above transformations, we get the following

\begin{cor}
\label{cor:pisigma} Let $G$ be a connected graph and $u\in V(G)$\,. Assume that $G_{1}$ is the
graph obtained from $G$ by attaching a tree $T$ ($T \not \cong P_{k}$ and $T \not \cong S_{k}$) of
order $k$ at $u$\,; $G_{2}$ is the graph obtained from $G$ by identifying $u$ with an endvertex of
a path $P_{k}$\,; $G_{3}$ is the graph obtained from $G$ by identifying $u$ with the center of a
star $S_{k}$\,. Then
$$\xi(G_{2})<\xi(G_{1})< \xi(G_{3})\,.$$
\end{cor}

\section{Upper and lower bounds}

For positive integers $n$ and $m$\,, with $n - 1 \leq m < \frac{n(n-1)}{2}$\,, let
$$
a = a_{n,m} = \left \lfloor \frac{2n-1-\sqrt{(2n-1)^2-8m}}{2} \right \rfloor\,.
$$
Note that $a$ is the largest integer satisfying $f (a) \geq 0$ for the quadratic function $f (a) =
a^2 - 2na + a + 2m$\,. Let $G_{n, m}$ be the set of graphs $K_a \vee H$\,, where $H$ is a graph
with $n - a$ vertices and $m - \frac{a(a-1)}{2}- a (n - a)$ edges.

\begin{thm}
\label{thm:lower} Let $G$ be a graph with $n$ vertices and $m$ edges, $n - 1 \leq m <
\frac{n(n-1)}{2}$\,. Then
$$
\xi(G) \geq 4m - a (n - 1)\,,
$$
with equality if and only if $G \in G_{n, m}$\,.
\end{thm}

\begin{proof}
It is obvious that $a \geq 1$\,. If $r (G) \geq 2$\,, then
\begin{equation}
\label{eq:lower} \xi(G) = \sum_{v \in V (G)} deg (v) \cdot \varepsilon (v) \geq r (G) \cdot \sum_{v
\in V (G)} deg (v) \geq 2 \cdot 2m > 4m - a (n - 1)\,.
\end{equation}
For $r (G) = 1$\,, let $k$ be the number of vertices of degree $n - 1$\,. It follows that the
remaining $n - k$ vertices have degree at least $k$\,. The inequality $2m \geq (n - 1)k + k (n -
k)$ implies $k \leq a$\,. Therefore, $\xi(G) \geq 4m - a (n - 1)$\,, with equality if and only if
$G$ has exactly $a$ vertices of degree $n - 1$ and all other vertices have eccentricity two, i.e.
$G \in G_{n, m}$\,.
\end{proof}

For $m = n$ and $m = n + 1$\,, we have $a = 1$\,.

\begin{cor}
Let $G$ be a unicyclic graph on $n \geq 3$ vertices. Then $\xi^c (G) \geq 3n+1$\,, with equality if
and only if $G$ is formed by adding one edge to the star $S_n$\,.
\end{cor}

\begin{cor}
Let $G$ be a bicyclic graph on $n \geq 5$ vertices. Then $\xi^c (G) \geq 3n+5$\,, with equality if
and only if $G$ is formed by adding two edges to the star $S_n$\,.
\end{cor}

Let $K_n - ke$ be the graph obtained by deleting $k$\,, $0 \leq k \leq \lfloor n/2 \rfloor$\,,
independent edges from the complete graph $K_n$ ($K_n - 0e \cong K_n$). The first Zagreb index
\cite{GuDa04} is defined as $M_1 (G) = \sum_{v \in V (G)} deg^2 (v)$\,.

\begin{thm}
Let $G$ be a connected graph on $n \geq 3$ vertices and $m$ edges. Then
$$
\xi(G) \leq 2nm - M_1 (G)\,,
$$
with equality if and only if $G \cong K_n - ke$\,, for $k = 0, 1, \ldots, \lfloor n/2 \rfloor$\,,
or $G \cong P_4$\,.
\end{thm}

\begin{proof}
Let $d_i (v)$ be the number of vertices at distance $i$ from the vertex $v$\,. It can be easily
seen that $\varepsilon (v) \leq n - deg (v)$\,. The equality is achieved for $\varepsilon (v) = 1$
and $deg (v) = n - 1$\,, or $\varepsilon (v) \geq 2$ and $d_2 (v) = d_3 (v) = \ldots =
d_{\varepsilon (v)} (v) = 1$\,. Further, we get
$$
\xi^c (G) = \sum_{v \in V} deg (v) \cdot \varepsilon (v) \leq \sum_{i = 1}^n deg (v) \cdot (n - deg
(v)) = 2nm - M_1 (G)\,.
$$

Suppose that equality holds in the above inequality. If $\varepsilon (u) = 1$ for some $u \in V
(G)$\,, then $deg (v) = n - 1$ and $\varepsilon (v) \leq 2$ for all $v \neq u$\,. If $\varepsilon
(v) = 1$ for all $v \in V (G)$\,, then $G \cong K_n$\,. Suppose that $\varepsilon (v) = 2$ for some
$v \in V (G)$\,. Then, there exist a vertex $w \in V (G)$\,, such that $d (v, w) = 2$\,. Since $d_2
(v) = d_2 (w) = 1$\,, the vertex $v$ is unique for fixed $v$\,, and it follows that $deg (v) = deg
(w) = n - 2$\,. This implies that $G \cong K_n - ke$\,, for $k = 1, 2, \ldots \lfloor (n-1)/2
\rfloor$\,.

In the other case, $\varepsilon (u) \geq 2$ for all $u \in V (G)$\,. If $\varepsilon (v) = 2$ for
all $v \in V$\,, then $deg (v) = n - 2$ for all $v \in V (G)$\,, implying that $G \cong K_n -
\frac{n}{2} e$ (with even $n$). If $\varepsilon (v) \geq 3$ for some vertex $v$\,, then the
diameter of $G$ is equal to 3 (otherwise, the center vertex would have at least two neighbors at
distance two), and then it follows that $G \cong P_4$\,.

It can be easily seen that the upper bound for $\xi^c (G)$ is attained for $G \cong K_n - ke$\,, $k
= 0, 1, \ldots, \lfloor n/2 \rfloor$\,, or $G \cong P_4$\,.
\end{proof}

The parameter $D'(G)$ is called the {\it degree distance} of $G$ and it was introduced by Dobrynin
and Kochetova \cite{DoKo94} and Gutman \cite{Gu94} as a weighted version of the Wiener index
\cite{DoEG01}
\begin{equation}
D'(G) =\sum_{{u,v}\in V(G)} (deg (u) + deg (v)) \cdot d(u,v)= \sum_{v \in V} deg (v) \cdot D (v)\,,
\end{equation}
where $D (v)=\sum_{u \in V} d (u, v)$\,. For further details on degree distance index see
\cite{DaGuMuSw09,IlKlSt10}.

\begin{thm}
Let $G$ be a connected graph on $n \geq 2$ vertices. Then
$$
\xi^c (G) \geq \frac{1}{n - 1} D' (G)\,,
$$
with equality if and only if $G \cong K_n$\,.
\end{thm}

\begin{proof}
Obviously, $\varepsilon (u) \geq \frac{D (u)}{n - 1}$ with equality if and only if $d (u, v)$ is
constant for all vertices $v \neq u$\,. It follows
$$
\xi^c (G) \geq \sum_{u \in V (G)} deg (u) \frac{ D (u) }{n - 1} = \frac{1}{n - 1} D' (G)\,,
$$
with equality if and only if $G$ is a complete graph.
\end{proof}

\subsection{Maximum and minimum values}

\begin{thm}
Let $G$ be $n$-vertex connected graph, with $n \geq 4$\,. Then $\xi^c (G) \geq 3 (n - 1)$\,, with
equality if and only if $G \cong S_n$\,.
\end{thm}

\begin{proof}
We will use the same notation as in the proof of Theorem \ref{thm:lower}. If $k = 0$\,, then
according to \eqref{eq:lower} we have
$$
\xi^c (G) \geq 4m > 3 (n - 1)\,.
$$
Now suppose that $k \geq 1$\,. Using inequality $2m \geq (n - 1)k + k (n - k)$\,, we get
$$
\xi^c (G) \geq 2 ((n - 1)k + k (n - k))- k(n-1) = k(3n - 2k - 1)\,.
$$
The quadratic function $f (x) = x (3n - 2x - 1)$ is increasing for $1 \leq x \leq \frac{3n-1}{4}$
and decreasing for $\frac{3n-1}{4} \leq x \leq n$\,. Therefore, the minimum value of $f (x)$ is
attained for $x = 1$ or $x = n$\,. Obviously, $f (1) = 3 (n - 1) < n (n - 1) = f (n)$ and $\xi^c
(G) \geq 3 (n - 1)$ with equality if and only if $k = 1$ and $m = n - 1$\,, i.e. $G \cong S_n$\,.
\end{proof}

The lollipop graph $LP_{n,d}$ is obtained from a complete graph $K_{n - d}$ and a path $P_d$\,, by
joining one of the end vertices of $P_d$ to one vertex of $K_{n-d}$ (see Fig. 1).
$$
\xi^c (LP_{n, d}) = \left\{
\begin{array}{l}
\frac{1}{2} \left (2 - 2d + d^2 +2d^3 -2n +2dn - 4d^2n+2dn^2\right) \quad \qquad \mbox{for even $d$}\,;\\
\frac{1}{2} \left (3 - 2d + d^2 +2d^3 -2n +2dn - 4d^2n+2dn^2\right) \quad \qquad \mbox{for odd
$d$}\,.
\end{array}
\right.
$$

\begin{figure}[h]
  \center
  \includegraphics [width = 8.5cm]{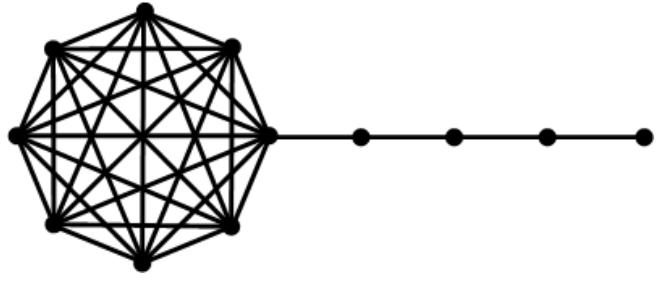}
  \caption { \textit{ The lolipop graph $LP_{12,4}$\,.} }
\end{figure}

The following result is presented in \cite{MoMuSw09} and \cite{DoSaVu10}.

\begin{thm}\label{thm:max}
Let $G$ be a connected graph of order $n$ and diameter $d$\,. Then
$$\xi^c (G) \leq d (n - d)^2 + O (n^2)\,,
$$
and this bound is the best possible.
\end{thm}

\begin{proof}
Let $P = u_0 u_1 \ldots u_d$ be a diametral path, and let $M \subset V$ be the set of the remaining
vertices which are not on $P$\,.

Partition the vertices of $P$ in three sets $V_0, V_1, V_2$\,, such that the set $V_i$ contains the
vertices $u_j$ with $i \equiv j \pmod 3$\,, $i = 0, 1, 2$\,. Let $x$ and $y$ be two vertices from
the set $V_i$\,. Since the distance between $x$ and $y$ is at least three, $x$ and $y$ cannot have
a common neighbor. It follows that the set of neighbors of all vertices from $V_i$ are disjoint,
and consecutively $\sum_{x \in V_i} deg (x) \leq n - |V_i|$\,, for $i = 0, 1, 2$\,.
\begin{eqnarray*}
\sum_{x \in V (P)} deg (x) \cdot \varepsilon (x) &=& \sum_{x \in V_0} deg (x) \cdot \varepsilon (x)
+ \sum_{x \in V_1} deg (x) \cdot \varepsilon (x) + \sum_{x \in V_2} deg (x) \cdot \varepsilon
(x) \\
&\leq& d \left ( \sum_{x \in V_0} deg (x) + \sum_{x \in V_1} deg (x) + \sum_{x \in V_2} deg (x)
\right ) \\
&\leq& d (n - |V_0| + n - |V_1| + n - |V_2|) = d (3n - d + 1)\,.
\end{eqnarray*}

If a vertex $v \in M$ is connected with more than three vertices of $P$\,, then the distance
between the end-vertices of $P$ is less than $d$\,, a contradiction with the fact that $P$ is
diametral. Moreover, if a vertex $v \in M$ is connected with three vertices of $P$\,, then these
vertices must be consecutive. It follows that each vertex in $M$ is connected with at most three
vertices of $P$ and $deg (v) \leq n - d + 1$\,.
$$
\sum_{x \in M} deg (v) \cdot \varepsilon (v) \leq (n - d + 1) \cdot (n - d + 1) \cdot d\,.
$$

Finally, combining these sums we have
\begin{eqnarray*}
\xi^c (G) &=& \sum_{x \in V (P)} deg (x) \cdot \varepsilon (x) + \sum_{x \in M} deg (v) \cdot
\varepsilon (v) \\
&\leq& d (3n - d + 1) + (n - d + 1)^2 d \\
&=& d (n - d)^2 + O (n^2)\,,
\end{eqnarray*}
which completes the proof.
\end{proof}

From arithmetic-geometric mean inequality, we have
$$
d \cdot \frac{n - d}{2} \cdot \frac{n - d}{2} \leq \left (\frac{d + \frac{n-d}{2} + \frac{n
-d}{2}}{3} \right)^3 = \frac{n^3}{27}.
$$

Therefore, the maximum of $d (n - d)^2$ is achieved for $d = \frac{n - d}{2}$\,, i.e. for $d =
\lfloor n / 3 \rfloor$\,. The lollipop graph $LP_{n, \lfloor n / 3 \rfloor}$ shows that the bound
in Theorem \ref{thm:max} is the best possible.

\section{Eccentric connectivity index of trees}

For an arbitrary tree $T$ on $n$ vertices it holds \cite{ZhDu09},
$$
\left \lfloor \frac{3 (n - 1)^2 + 1}{2} \right \rfloor = \xi^c (P_n) \geq \xi^c (T) \geq \xi^c
(S_n) = 3 (n - 1) \ .
$$

The starlike tree $T(n_1,n_2,\ldots,n_k)$ is a tree composed of the root $v$\,, and paths
$P_{n_1}$\,, $P_{n_2}$\,, \ldots, $P_{n_k}$\,, attached at $v$\,. The number of vertices of
$T(n_1,n_2,\ldots,n_k)$ is equal to $n_1 + n_2 + \cdots + n_k + 1$\,. The starlike tree $SB_{n, p}
= T(n_1, n_2, \ldots, n_p)$ is {\em balanced} if all paths have almost equal lengths, i.e., $|n_i -
n_j| \leq 1$ for every $1 \leq i \leq j \leq p$\,.

Our goal here is also to add some further evidence to support the use of $\xi^c$ as a measure of
branching in alkanes. While the measure of branching cannot be formally defined, there are several
properties that any proposed measure $T$ has to satisfy \cite{770},
$$
TI (P_n) < TI (X_n) < TI (S_n) \qquad \mbox{or} \qquad TI (P_n) > TI (X_n) > TI (S_n).
$$

Using Theorem \ref{thm-pi}, we have the following chain of inequalities
\begin{equation}
\label{eq:order-balanced} \xi^c (P_{n}) = \xi^c (BS_{n,2}) > \xi^c(BS_{n,3}) > \ldots >
\xi^c(BS_{n,n-2}) > \xi^c(BS_{n,n-1}) = \xi^c (S_{n})\,.
\end{equation}

\subsection{Trees with given number of pendent vertices}

Let $T^{(n, p)}$ the set of trees obtained by attaching $a$ and $p - a$ pendent vertices
respectively to the end vertices of the path $P_{n - p}$\,, for $1 \leq a \leq \lfloor \frac{p}{2}
\rfloor$\,. In \cite{ZhDu09} the authors proved that among $n$-vertex trees with $p$ pendent
vertices, $2 \leq p \leq n - 1$\,, the maximum eccentric connectivity index is achieved exactly of
the trees in $T^{(n, p)}$\,,
$$
\xi^c (T) \leq \left \lfloor \frac{3(n-p+1)^2+1}{2} \right \rfloor + (p - 2)(2n-2p + 1)\,.
$$

Here we determine the $n$-vertex trees with $2 \leq p \leq n - 1$ pendent vertices that have
minimum eccentric connectivity index. In was proven in \cite{IlIl09} that $SB_{n, p}$ has minimal
Laplacian coefficients among $n$-vertex with $k$ pendent vertices. In particular, this extremal
balanced starlike tree $BS_{n,k}$ minimizes the Wiener index and the Laplacian-like energy
\cite{IlKrIl09}.

\begin{thm}
\label{thm:pendent} The balanced $p$-starlike tree $SB_{n, p}$ has minimum eccentric connectivity
index among trees with $p$ pendent vertices, $2 < p < n - 1$\,.
\end{thm}

\begin{proof}
Let $T$ be a rooted $n$-vertex tree with $p$ pendent vertices. If $T$ contains only one branch
vertex of degree greater than two, we can apply Theorem \ref{thm-pi} in order to get balanced
starlike tree $SB_{n, p}$\,, while do not change the number of pendent vertices. If there are
multiple vertices with degree greater than 2, such that there are only pendent paths attached below
them -- we take the one furthest from the center vertex of the tree. By repetitive application of
$\delta$ transformation and balancing pendent paths, the eccentric connectivity index decreases.

Assume that we arrived at tree with two centers $C = \{v, w\}$ with only pendent paths attached at
both centers. If all pendent paths have equal lengths, then $n = k p + 2$\,. Since we can reattach
$p - 2$ pendent paths at any central vertex and do not change $\xi^c (T)$ -- it follows that there
are exactly $\lfloor \frac{p}{2} \rfloor$ extremal trees with minimum eccentric connectivity index
in this special case.

Now, let $R$ be the path with length $r = r (T) - 1$ attached at $v$ and let $Q$ be the shortest
path of length $q$ attached at $w$\,. After applying $\delta$ transformation at vertex $v$\,, the
eccentric connectivity index remains the same. If we apply the transformation from Theorem
\ref{thm-pi} at two pendent paths of lengths $r + 1$ and $q$ attached at $w$\,, we will strictly
decrease the eccentric connectivity index. Finally, we conclude that $SB_{n, p}$ is the unique
extremal tree that minimizes $\xi^c$ among $n$-vertex trees with $p$ pendent vertices for $n \not
\equiv 2 \pmod p$\,.
\end{proof}

\subsection{Trees with bounded vertex degree}

Chemical trees (trees with maximum vertex degree at most four) provide the graph representations of
alkanes \cite{GuPo86}. It is therefore a natural problem to study trees with bounded maximum
degree. Denote by $\Delta = \Delta(T)$ the maximum vertex degree of a tree $T$\,. The path $P_n$ is
the unique tree with $\Delta = 2$\,, while the star $S_n$ is the unique tree with $\Delta = n-1$\,.
Therefore, we can assume that $3 \leq \Delta \leq n - 2$\,.

The broom $B_{n, \Delta}$ is a tree consisting of a star $S_{\Delta + 1}$ and a path of length $n -
\Delta - 1$ attached to an arbitrary pendent vertex of the star (see Fig. 2). It is proven in
\cite{LiGu07} that among trees with maximum vertex degree equal to $\Delta$\,, the broom $B_{n,
\Delta}$ uniquely minimizes the largest eigenvalue of the adjacency matrix. In \cite{YaYe05} and
\cite{YuLv06} it was demonstrated that the broom has minimum energy among trees with, respectively,
fixed diameter and fixed number of pendent vertices. \vspace{0.2cm}

\begin{figure}[ht]
  \center
  \includegraphics [width = 8cm]{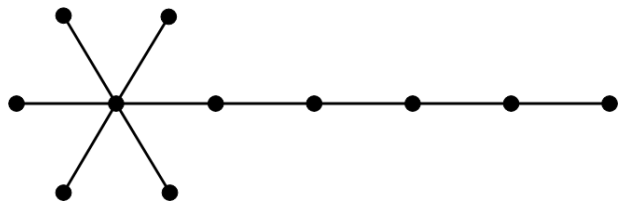}
  \caption { \textit{ The broom $B_{11, 6}$\,. } }
\end{figure}

Notice that the broom $B_{n, \Delta} = T (n - \Delta, 1, 1, \ldots, 1)$ is a $\Delta$-starlike
tree.

\begin{thm}
\label{thm-broom} Let $T \not \cong B_{n, \Delta}$ be an arbitrary tree on $n$ vertices with
maximum vertex degree $\Delta$\,. Then
$$
\xi^c (B_{n, \Delta}) > \xi^c (T) \ .
$$
\end{thm}

\begin{proof}
Fix a vertex $v$ of degree $\Delta$ as a root and let $T_1, T_2, \ldots, T_{\Delta}$ be the trees
attached at~$v$\,. We can repeatedly apply the transformation described in Theorem~\ref{thm-pi} at
any vertex of degree at least three with largest eccentricity from the root in every tree~$T_i$\,,
as long as $T_i$ does not become a path. When all trees $T_{1} ,T_{2},\dots, T_{\Delta}$ turn into
paths, we can again apply transformation from Theorem~\ref{thm-pi} at the vertex~$v$ as long as
there exists at least two paths of length greater than one, further decreasing the eccentric
connectivity index. Finally, we arrive at the broom $B_{n, \Delta}$ as the unique tree with maximum
eccentric connectivity index.
\end{proof}

By direct verification, it holds
$$
\xi^c (BT_{n, \Delta}) = \left \lfloor \frac{3n^2 - 2\Delta n - 2n - \Delta^2 +4\Delta}{2} \right
\rfloor .
$$

From the above proof, we also get that $B'_{n,\Delta} = T(n-\Delta-1,2,1,\ldots,1)$ has the second
minimal $\xi^c$ among trees with maximum vertex degree $\Delta$\,.

From Theorem~\ref{thm-broom} we know that the
maximum eccentric connectivity index among trees on $n$~vertices is achieved for one of the brooms
$B_{n,\Delta}$\,. If $\Delta>2$\,, we can apply the transformation from Theorem~\ref{thm-pi} at the
vertex of degree~$\Delta$ in $B_{n, \Delta}$ and obtain $B_{n, \Delta-1}$\,. Thus, it follows
$$
\xi^c (S_{n}) = \xi^c(B_{n,n-1}) < \xi^c(B_{n,n-2}) < \cdots < \xi^c(B_{n,3})<
\xi^c(B_{n,2})=\xi^c(P_{n}) \ .
$$

Also, it follows that $B_{n, 3}$ has the second maximum eccentric connectivity index ($\lfloor 3
(n-1)^2 / 2 \rfloor - n$) among trees on $n$ vertices.

\begin{thm}
\label{thm-rot} Let $T$ be a rooted tree, with a center vertex $c$ as root. Let $u$ be the vertex
closest to the root vertex, such that $deg (u) < \Delta$\,. Let $w$ be the pendent vertex most
distant from the root, adjacent to vertex $v$\,, such that $\varepsilon (v) > \varepsilon (u)$\,.
Construct a tree $T'$ by deleting the edge $vw$ and inserting the new edge $uw$\,. Then
$$
\xi^c (T) > \xi^c (T') \ .
$$
\end{thm}

\begin{proof}
In the transformation $T \to T'$ the degrees of vertices other than $u$ and $v$ remain the same,
while $deg' (u) = deg (u) + 1$ and $deg' (v) = deg (v) - 1$\,. Since the tree is rooted at the
center vertex, the radius of $T$ is equal to $r (T) = d (c, w)$\,. Furthermore, there exists a
vertex $w'$ in a different subtree attached to the center vertex, such that $d (c, w') = r (T)$ or
$d (c, w') = r (T) - 1$\,. From the condition $\varepsilon (v) > \varepsilon (u)$\,, it follows
that $d (c, w') > d (c, u)$ and $w' \neq u$\,.

By rotating the edge $vw$ to $uw$\,, the eccentricity of vertices other than $w$ decrease if and
only if $w$ is the only vertex at distance $r (T)$ from the center vertex. Otherwise the
eccentricities remain the same. In both cases,
\begin{eqnarray*}
\xi^c (T) - \xi^c (T') &\geq& deg (v)\,\varepsilon (v) + deg (w)\,\varepsilon (w)
+ deg (u)\,\varepsilon (u)\\
&-& \left[ deg' (v)\,\varepsilon' (v) + deg' (w)\,\varepsilon' (w) +
deg' (u)\,\varepsilon' (u) \right] \\
&\geq& \varepsilon (v) + (\varepsilon (v) - \varepsilon (u)) - \varepsilon (u) = 2 (\varepsilon (v)
- \varepsilon (u)) > 0 \ .
\end{eqnarray*}
This completes the proof.
\end{proof}

The Volkmann tree $VT (n, \Delta)$ is a tree on $n$ vertices and maximum vertex degree $\Delta$\,,
defined as follows \cite{770,FiHo02}. Start with the root having $\Delta$ children. Every vertex
different from the root, which is not in one of the last two levels, has exactly $\Delta -1$
children. In the last level, while not all vertices need to exist, the vertices that do exist fill
the level consecutively. Thus, at most one vertex on the level second to last has its degree
different from $\Delta$ and $1$\,. In \cite{770,FiHo02} it was shown that among trees with fixed
$n$ and $\Delta$\,, the Volkmann trees have minimum Wiener index. Volkmann trees have also other
extremal properties among trees with fixed $n$ and $\Delta$ \cite{GuFMG07,790,SiTo05,YuLu08}.

\begin{figure}[ht]
  \center
  \includegraphics [width = 6cm]{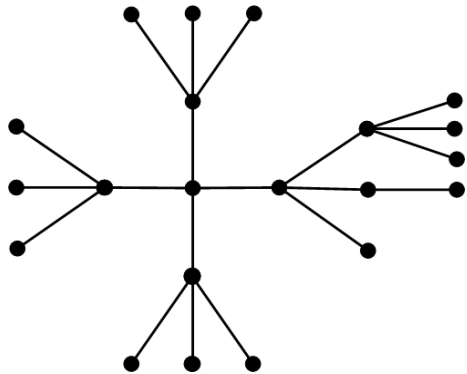}
  \caption { \textit{ The Volkmann tree $VT (21, 4)$\,. } }
\end{figure}

\begin{thm}
\label{thm-volkman} Let $T$ be an arbitrary tree on $n$ vertices with maximum vertex
degree~$\Delta$\,. Then
$$
\xi^c (T) \geq \xi^c (VT_{n, \Delta}).
$$
\end{thm}

\begin{proof}
Among $n$-vertex trees with maximum degree $\Delta$\,, let $T^*$ extremal tree with minimum
eccentric connectivity index. Assume that $u$ is a vertex closest to the root vertex $c$\,, such
that $deg (u) < \Delta$ and let $w$ be the pendent vertex most distant from the root, adjacent to
vertex $v$\,. Also, let $k$ be the greatest integer, such that
$$
n \geq 1 + \Delta + \Delta (\Delta - 1) + \Delta (\Delta - 1)^2 + \cdots + \Delta (\Delta - 1)^{k -
1}\ .
$$

First, we will show that the radius of $T^*$ has to be less than or equal to $k + 1$\,. Assume that
$r (T^*) = d (c, w) > k + 1$\,. Since the distance from the center vertex to $u$ is less than or
equal to $k$\,, it follows that
$$
\varepsilon (v) \geq 2 r (T^*) - 2 \geq k + r (T^*) \geq \varepsilon (u)\ .
$$
If strict inequality holds, we can apply Theorem \ref{thm-rot} and decrease the eccentric
connectivity index -- which is in contradiction with the assumption that $T^*$ is the tree with
minimum $\xi^c$\,. Therefore, $\varepsilon (v) = \varepsilon (u)$ and after performing
transformation from Theorem \ref{thm-rot} the eccentric connectivity index does not change.
According to the definition of number $k$\,, after finitely many transformations the vertex $w$
will be the only vertex at distance $r (T)$ from the center vertex and we will strictly decrease
$\xi^c (T^*)$\,. Also, this means that for the case $n = 1 + \Delta + \Delta (\Delta - 1) + \Delta
(\Delta - 1)^2 + \cdots + \Delta (\Delta - 1)^{k - 1}$\,, the Volkmann tree is the unique tree with
minimum eccentric connectivity index.

Now, we can assume that the radius of $T^*$ is equal $k + 1$\,. If the distance $d (c, u)$ is less
than $k - 1$\,, it follows again that $\varepsilon (v) > \varepsilon (u)$\,, which is impossible.
Therefore, the levels $1, 2, \ldots, k - 1$ are full (level $i$ contains exactly $\Delta (\Delta -
1)^{i - 1}$ vertices), while the $k$-th and $(k + 1)$-th level contain
$$
L = n - \left[ 1 + \Delta + \Delta (\Delta - 1) + \Delta (\Delta - 1)^2 + \cdots + \Delta (\Delta -
1)^{k - 1} \right]
$$
vertices.

Assume that $T^*$ has only one center vertex -- then $d (c, w) = k + 1$ and $\varepsilon (v) = 2 r
(T^*) - 1$\,. If $d (c, u) = k - 1$\,, we can apply transformation from Theorem \ref{thm-rot} and
strictly decrease the eccentric connectivity index. Thus, for $L > (\Delta - 1)^k$\,, the $k$-th
level is also full and the pendent vertices in $(k + 1)$-th level can be arbitrary assigned. Using
the same argument, for $L \leq (\Delta - 1)^k$ the extremal trees are bicentral. By completing the
$k$-th level, we do not change the eccentric connectivity index -- since $\varepsilon (v) =
\varepsilon (u)$\,. Finally, $\xi^c (T^*) = \xi^c (VT (n, \Delta))$ and the result follows.
\end{proof}

In Table 1 from \cite{IlGu10}, the authors presented the minimum values of eccentric connectivity
index among $n \leq 20$ vertex trees with maximum vertex degree $\Delta$\,, together with the
number of such extremal trees (of which one is the Volkman tree).

\subsection{Trees with given matching or independence number}

Two distinct edges in a graph $G$ are independent if they are not incident with a common vertex in
$G$\,. A set of pairwise independent edges in $G$ is called a matching in $G$\,, while a matching
of maximum cardinality is a maximum matching in $G$\,. The matching number $\beta(G)$ of $G$ is the
cardinality of a maximum matching of $G$\,. It is well known that $\beta(G) \leq \frac{n}{2}$, with
equality if and only if $G$ has a perfect matching. The independence number of $G$\,, denoted by
$\alpha(G)$\,, is the size of a maximum independent set of $G$\,.

If $\frac{n-1}{2} < m \leq n-1$\,, then $A_{n,m}$ is the tree obtained from $S_{m+1}$ by adding a
pendent edge to some of $n-m-1$ of the pendent vertices of $S_{m+1}$\,. We call $A_{n,m}$ a spur
(see Fig. 4). Clearly, $A_{n,m}$ has $n$ vertices and $m$ pendent vertices; the matching number,
independence number and domination number of $A_{n,m}$ are $n-m$\,, $m$ and  $n-m$\,, respectively.

\begin{figure}[h]
  \center
  \includegraphics [width = 5cm]{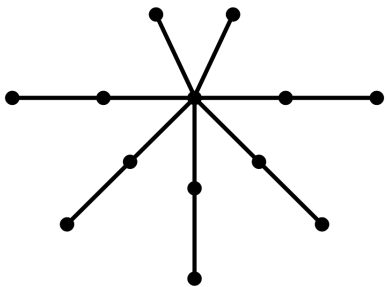}
  \caption { \textit{ The spur $A_{13, 7}$\,. } }
\end{figure}

It can be easily seen that
\begin{equation}
\label{eq:spur} \xi(A_{n,m}) = \left\{
\begin{array}{l}
3n-3, \qquad \quad \quad \mbox{if $m = n - 1$}\,; \\
5n-6, \qquad \qquad \mbox{if $m = n - 2$}\,; \\
7n-2m-7, \quad \ \mbox{if $\frac{n-1}{2} < m < n - 2$\,.} \\
\end{array}
\right.
\end{equation}

Double star $S_{a,b}$ is a tree obtained by joining the centers of two stars $S_a$ and $S_b$ with
an edge. The only tree with $\beta = 1$ is star $S_n$\,, while for $\beta = 2$ we have double stars
$S_{a,b}$ with $a + b = n$\,. It can be easily verified that all double stars have the same
eccentric connectivity index equal to $5n-6$\,.

\begin{thm}
\label{thm:matching} Let $T$ be a tree on $n$ vertices with matching number $\beta > 2$\,. Then
$$
\xi^c (T) \geq \xi^c(A_{n,n-\beta})\,,
$$
with equality holding if and only if $T\cong A_{n,n-\beta}$\,.
\end{thm}

\begin{proof}
Suppose that $T$ has $k$ pendent vertices. For every edge $e=uv$ in the maximum matching holds: at
most one of the vertices $v$ and $u$ has degree equal to one. Therefore,
$$
k\leq \beta+n-2\beta=n-\beta\,,
$$
and by Theorem \ref{thm:pendent}, we have $\xi^c(T)\geq \xi^c(BS_{n,k})$\,. Using Equation
(\ref{eq:order-balanced}), it follows $\xi^c(BS_{n,k}) \geq \xi^c(BS_{n,n-\beta}) =
\xi^c(A_{n,n-\beta})$ since $n-\beta \geq \frac{n}{2}$\,. Finally, $\xi^c(T)\geq
\xi^c(A_{n,n-\beta})$\,, with equality holding if and only if $T\cong A_{n,n-\beta}$\,.
\end{proof}

By Theorem \ref{thm:matching}, we have the following
\begin{cor}
Let $T$ be a tree of order $n$ with perfect matching. Then
$$
\xi^c(T) \geq 6n - 7,
$$
with equality holding if and only if $T\cong A_{n,n/2}$\,.
\end{cor}

\begin{thm}\label{9}
Let $T$ be a tree on $n$ vertices with independence number $\alpha$\,. Then
$$
\xi^c(T)\geq \xi^c(A_{n, \alpha}),
$$
with equality holding if and only if $T \cong A_{n,\alpha}$ for $\alpha \neq n - 2$ and $T \cong
S_{a,b}$ for $\alpha = n - 2$\,.
\end{thm}

\begin{proof}
Since all the pendent vertices form an independent set, it follows $k\leq \alpha$\,. Every tree is
a bipartite graph, and we get $\alpha \geq \lceil\frac{n}{2}\rceil$\,. By Theorem
\ref{thm:pendent}, we have $\xi^c (T)\geq \xi^c(BS_{n,k})$ and
$$
\xi^c(BS_{n,k})\geq \xi^c(BS_{n,\alpha})=\xi^c(A_{n,\alpha})\,. $$
Therefore, $\xi^c(T)\geq
\xi^c(A_{n,\alpha})$\,, with equality holding if and only if $T\cong A_{n,\alpha}$ for $\alpha \neq
n - 2$ and $T \cong S_{a,b}$ for $\alpha = n - 2$\,.
\end{proof}

\subsection{Trees with given diameter or radius}

The vertices of minimum eccentricity form the center of a graph. A tree has exactly one or two
adjacent center vertices; in this latter case one speaks of a bicenter. In what follows, if a tree
has a bicenter, then our considerations apply to any of its center vertices.

For a tree $T$ with radius $r(T)$\,,
$$
\label{tree_center} d (T) = \left\{
\begin{array}{l l}
  2\,r(T) - 1 & \quad \mbox{if $T$ has a bicenter }\\[3mm]
  2\,r (T) & \quad \mbox{if $T$ has has a center. }\\
\end{array} \right.
$$

Let $C_{n,d}(p_1, p_2, \ldots, p_{d-1})$ be a caterpillar with $n$ vertices obtained from a path
$P_{d+1} = v_0v_1\ldots v_{d-1}v_d$ by attaching $p_i \geq 0$ pendent vertices to $v_i$\,, $1\leq i
\leq d-1$\,, where $n=d+1+\sum_{i=1}^{d-1}p_i$\,. Denote
$$
C_{n,d,i}=C_{n,d}(\underbrace{0, \ldots, 0}_{i-1}, n-d-1, 0, \ldots ,0)\,.
$$
Let $T_{(n, d)}$ be the set of $n$-vertex trees obtained from
path $P_{d + 1} = v_0 v_1 \ldots v_d$ by attaching $n - d - 1$ pendent vertices to $v_{\lfloor d/2
\rfloor}$ and/or $v_{\lceil d/2 \rceil}$\,, where $2 \leq d \leq n - 2$\,. Note that there is only
one tree in the set $T_{(n,d)}$ for even $d$\,, and $\lfloor (n - d + 1)/2 \rfloor$ trees for odd
$d$\,. Zhou and Du in \cite{ZhDu09} proved the following
\begin{thm}
Let $T$ be an arbitrary tree on $n$ vertices and diameter $d$\,. Then
$$
\xi^c (T) \geq \xi (T^*)\,, \quad T^* \in  T_{(n, d)}
$$
with equality if and only if $T \in T_{(n, d)}$\,.
\end{thm}

Here we present an alternative proof of this theorem.

\begin{proof}
Let $T$ be $n$-vertex tree with diameter $d$ having minimal eccentric connectivity index. Let
$P_{d+1}=v_0v_1\ldots v_{d-1}v_d$ be a path of length $d$\,. By Corollary \ref{cor:pisigma}, all
trees attached to the path $P_{d+1}$ must be stars, which implies that $T \cong C_{n,d}(p_1, p_2,
\ldots , p_{d-1})$\,. Let $i$ be the minimal index such that $p_i > 0$ and let $j$ be the maximal
index such that $p_j > 0$\,. Without loss of generality assume that $i \leq d - j$\,. By applying
the transformation from Theorem \ref{thm-delta} at the vertices $v_i$ in $C_{n,d}(p_1, p_2, \ldots,
p_{d-1})$\,, we get that exactly trees from $T_{(n, d)}$ have minimal eccentric connectivity index.
Note that we cannot apply $\delta$ transformation if the vertex $v$ is a central vertex, which is
exactly the case for trees in $T_{(n, d)}$\,. This completes the proof.
\end{proof}

It can be easily computed that
$$
\xi^c (C_{n,d, \lfloor d / 2\rfloor}) = \left \lfloor \frac{3d^2+1}{2} \right \rfloor + (n - d - 1)
\left (1 + 2 \left \lceil \frac{d}{2} \right \rceil \right ).
$$

Using transformation from Theorem \ref{thm-pi} applied to a center vertex, it follows that $\xi^c
(T') < \xi^c (T'')$ for $T' \in T_{(n, 2r-1)}$ and $T'' \in T_{(n, 2r)}$\,.

\begin{cor}
Let $T$ be an arbitrary tree on $n$ vertices with radius $r$\,. Then
$$
\xi^c (T) \geq \xi^c (T_{(n, 2r-1)})\,,
$$
with equality if and only if $T \in T_{(n, 2r-1)}$\,.
\end{cor}

In addition, the following chain of inequalities hold
$$
\xi^c (P_{n}) = \xi^c(C_{n,n-1,\lfloor (n-1)/2 \rfloor}) > \xi^c(C_{n,n-2,\lfloor (n-2)/2 \rfloor})
> \ldots > \xi^c(C_{n,3,1}) > \xi^c(C_{n,2,1}) = \xi^c (S_{n}).
$$

It follows that $n$-vertex double stars have the second maximum eccentric connectivity index among
trees on $n$ vertices.

\section{Eccentric connectivity index of unicyclic graphs}

In this section we present the results from \cite{Il10} concerning the extremal unicyclic graphs
with $n$ vertices and given girth $k$ (see Fig. 5).

Let ${\cal{U}}_{n,k}$ be the set of all unicyclic graphs of order $n \geq 3$ with girth $k\geq
3$\,. By $L_{n,k}$ we denote the graph obtained from $C_{k}$ and $P_{n-k+1}$ by identifying a
vertex of $C_{k}$ with one end vertex of $P_{n-k+1}$\,. We denote by $H_{n,k}$ the graph obtained
from $C_{k}$ by adding $n-k$ pendent vertices to a vertex of $C_{k}$\,. For $U \in
{\cal{U}}_{n,k}$\,, if $k=n$ then $U \cong C_{k}$\,; if $k=n-1$ then $U \cong L_{n,n-1}$\,. So in
the following, we assume that $3\leq k \leq n-2$\,.  Ili\' c and Stevanovi\' c in \cite{StIl09}
proved that among all connected unicyclic graphs $C_n$ maximizes, while $H_{n,3}$ minimizes all
Laplacian coefficients.

\begin{figure}[h]
  \center
  \includegraphics [width = 12cm]{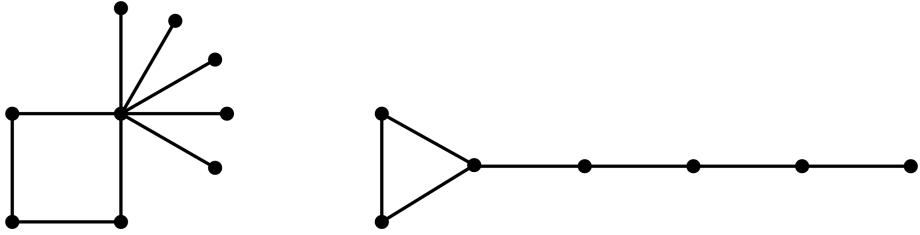}
  \caption { \textit{ Extremal unicyclic graphs $H_{9,4}$ and $L_{7,3}$\,.} }
\end{figure}

\subsection{The minimum case}

\begin{thm}
Let $G$ be a connected unicyclic graphs on $n \geq 5$ vertices with girth $k$\,. Then
$$
\xi^c (G) \geq \xi^c (H_{n,k})\,,
$$
with equality if and only if $G \cong H_{n,k}$\,.
\end{thm}

\begin{proof}
Using the transformation from Corollary \ref{cor:pisigma}, we can deduce that among $n$-vertex
unicyclic graphs with given girth $k$\,, the minimum of $\xi^c (G)$ is achieved exactly for graphs
$H (n, k; n_1, n_2, \ldots, n_k)$\,, which are obtained from a cycle $C = v_1v_2\ldots v_k$ with
$n_i$ pendent vertices attached at $v_i$ $(i = 1, 2, \ldots, k)$\,. Furthermore, it holds $n - k =
\sum_{i = 1}^k n_i$\,.

The eccentricity of the cycle vertices $v_i$ are exactly $\lfloor k/2 \rfloor$ or $\lfloor k / 2
\rfloor + 1$\,, depending whether the opposite vertex on the cycle has attached pendent vertices.
It follows that
\begin{eqnarray*}
\xi^c (G) &=& \sum_{i = 1}^k (n_i + 2) \cdot \varepsilon (v_i) + \sum_{i = 1}^k n_i \cdot
(\varepsilon (v_i) + 1) \\
&=& n - k + 2 \sum_{i = 1}^k (n_i + 1) \cdot \varepsilon (v_i)\,.
\end{eqnarray*}

For even $k$\,, there exist at least one vertex $v$ with $\varepsilon (v) = \lfloor k / 2 \rfloor +
1$; while for odd $k$ there are at least two vertices $v$ with $\varepsilon (v) = \lfloor k / 2
\rfloor + 1$\,. Finally, the eccentric connectivity index of $G$ is minimal exactly for $H_{n,k} =
H (n, k; n - k, 0, 0, \ldots, 0)$\,, by collecting all pendent vertices at one vertex from the
cycle $C$ (all but one or two vertices from the cycle have eccentricity $\lfloor k / 2 \rfloor$).
\end{proof}

For even $3 \leq k < n$ we have
$$
\xi^c (H_{n, k}) = n - k + 2 \left ( \frac{k}{2} + 1 + (n - k + 1) \frac{k}{2} + (k - 2)
\frac{k}{2} \right) = nk + n - k + 2\,,
$$
while for odd $3 \leq k < n$ we have
$$
\xi^c (H_{n, k}) = n - k + 2 \left ( 2 \cdot \frac{k + 1}{2} + (n - k + 1) \frac{k - 1}{2} + (k -
3) \frac{k - 1}{2} \right) = nk - k + 4.
$$

\subsection{The maximum case}

Using the transformations from Corollary \ref{cor:pisigma}, we can deduce that among $n$-vertex
unicyclic graphs with given girth, the maximum of $\xi^c (G)$ is achieved exactly for graphs $L (n,
k; n_1, n_2, \ldots, n_k)$\,, which are obtained from a cycle $C = v_1v_2\ldots v_k$ with paths
$P_i$ of length $n_i$ attached at $v_i$ $(i = 1, 2, \ldots, k)$\,. Furthermore, it holds $n - k =
\sum_{i = 1}^k n_i$\,.

Let $w_i$ denote the end vertex of $P_i$ (if $n_i = 0$ we have $w_i \equiv v_i$). Let the
eccentricity of $v_i$ in the graph $G \setminus P_i$\,, obtained by removing the vertices of the
pendent path $P_i$\,, be $m_i$\,. Only vertices that do not have degree equal to $2$ are $v_i$ and
$w_i$\,.
\begin{eqnarray*}
\xi^c (G) &=& 2 \sum_{v \in V (G)} \varepsilon (v) + \sum_{i = 1}^k \left (\varepsilon (v_i) -
\varepsilon (w_i) \right )\,.
\end{eqnarray*}

Notice that above formula is valid also for $v_i \equiv w_i$\,. It holds $\varepsilon (v_i) -
\varepsilon (w_i) = \max (m_i, n_i) - (n_i + m_i) = -\min (n_i, m_i)$\,, and finally
$$
\xi^c (G) = 2 \sum_{v \in V (G)} \varepsilon (v) - \sum_{i = 1}^k \min (m_i, n_i)\,.
$$

{\bf Claim 1. } $\min (n - k, \lfloor \frac{k}{2} \rfloor) \leq \sum_{i = 1}^k \min (m_i, n_i)$\,.

If $n - k \leq \lfloor k/2 \rfloor$\,, it follows $n_i \leq m_i$\,, and $n - k \leq \sum_{i = 1}^k
n_i$\,. If $\lfloor k/2 \rfloor < n - k$\,, using $m_i \geq \lfloor k/2 \rfloor$ it follows $n_i
\leq m_i$ for $i = 1, 2, \ldots, k$\,, and finally $\lfloor k/2 \rfloor \leq \sum_{i = 1}^k n_i = n
- k$\,. \hfill \QED

\medskip

We will prove that $\zeta (G) = \sum_{v \in V (G)} \varepsilon (v)$ is maximized exactly for $L_{n,
k}$\,. Let $k_i$ denote the number of vertices such that the vertex $w_i$ is the furthest for those
vertices.

\medskip

{\bf Claim 2. } For at most one pendent path $P_i$ with length $n_i > 0$\,, there are vertices from
$P_i$ different from $v_i$\,, such that $w_i$ is the furthest vertex in $G$\,.

Assume that there are two vertices $x \in P_i$ and $y \in P_j$\,, such that $ecc (x) = d (x, w_i)$
and $ecc (y) = d (y, w_j)$\,. It follows that $d (x, w_i) \geq d (x, y) + d (y, w_j) > d (y, w_j)$
and $d (y, w_j) \geq d (y, x) + d (x, w_i) > d (x, w_i)$\,, which is a contradiction. \hfill \QED

\medskip

We will refer to this path (if exists) as 'ecc path'. Assume that $P_i$ and $P_j$ are not 'ecc
paths'. By rotating the path $P_j$ on the path $P_i$\,, we mean the following transformation
$$
G = L (n, k; n_1, n_2, \ldots, n_i, \ldots, n_j, \ldots, n_k) \rightarrow L (n, k; n_1, n_2,
\ldots, n_i + n_j, \ldots, 0, \ldots, n_k) = G'\,.
$$

For at least $k_i$ vertices $v$\,, the eccentricity $\varepsilon (v)$ is increased by $n_j$\,,
while for at most $k_j$ vertices $v$\,, the eccentricity $\varepsilon (v)$ is decreased by $n_j$\,.
The eccentricity of vertices from $P_j$ is changed by $n_i + m_i - m_j$\,. Finally, we get
$$
\zeta (G') - \zeta (G) = n_j (k_i - k_j) + n_j (n_i + m_i - m_j) = n_j (k_i + m_i - k_j - m_j) +
n_i n_j\,.
$$

Without loss of generality, assume that $k_i + m_i \geq k_j + m_j$\,, and it follows $\zeta (G') >
\zeta (G)$\,.

By applying this argument, we increase in each step the total eccentricity and finally get either
$L_{n,k}$ or a unicyclic graph with two pendent paths (the longer one being 'ecc path'). By
rotating the shorter path on the longer path, it can be easily shown that $\zeta (L_{n,k}) > \zeta
(G)$\,.

Finally, we get the following
\begin{thm}
Let $G$ be a connected unicyclic graph on $n \geq 5$ vertices with girth $k$\,. Then
$$
\xi^c (G) \leq \xi^c (L_{n,k})\,,
$$
with equality if and only if $G \cong L_{n,k}$.
\end{thm}

\section{Composite graphs}

Many interesting graphs are composed of simpler graphs via several operations (also known as graph
products). We are interested in the type of relationship that exist between the eccentric
connectivity index of composite graphs and their components.

The {\em Cartesian product} $G_1\cp G_2 \cp \cdots\cp G_k$ of graphs $G_1,G_2,\ldots,G_k$ has the
vertex set $V(G_1)\times V(G_2) \times \ldots\times V(G_k)$\,, two vertices $(u_1,u_2\ldots,u_k)$
and $(v_1,v_2,\ldots,v_k)$ being adjacent if they differ in exactly one position, say in $i$-th,
and $u_iv_i$ is an edge of $G_i$\,. It is well-known (see \cite{ImKl00}) that for $G=G_1\cp G_2 \cp
\cdots\cp G_k$ and vertices $u, v\in G$ we have
$$
d_G(u,v)=\sum_{i=1}^k d_{G_i}(u_i,v_i)\,.
$$

Using the following relations for the eccentricity and the vertex degrees
$$
\varepsilon_{G_1 \cp G_2} (u_1, u_2) = \varepsilon_{G_1} (u_1) + \varepsilon_{G_2} (u_2)
$$
$$
deg_{G_1 \cp G_2} (u_1, u_2) = deg_{G_1} (u_1) + deg_{G_2} (u_2),
$$
we get
\begin{eqnarray*}
\xi^c (G_1 \cp G_2) &=& \sum_{(u_1, u_2) \in V (G_1 \cp G_2)} deg_{G_1 \cp G_2} (u_1, u_2) \cdot
\varepsilon_{G_1 \cp G_2} (u_1, u_2) \\
&=& \sum_{u_2 \in V (G_2)} \sum_{u_1 \in V (G_1)} deg_{G_1} (u_1) \cdot \varepsilon_{G_1} (u_1) +
\sum_{u_1 \in V (G_1)} \sum_{u_2 \in V (G_2)} deg_{G_2} (u_2) \cdot \varepsilon_{G_2} (u_2) \\
&\phantom{=}& + \ \sum_{u_1 \in V (G_1)} deg_{G_1} (u_1) \sum_{u_2 \in V (G_2)} \varepsilon_{G_2}
(u_2) + \sum_{u_2 \in V (G_2)} deg_{G_2} (u_2) \sum_{u_1 \in V (G_1)} \varepsilon_{G_1} (u_1)\\
&=& |G_2| \cdot \xi^c (G_1) + |G_1| \cdot \xi^c (G_2) + 2 \cdot \|G_2 \| \cdot \zeta (G_1) + 2
\cdot \|G_1 \| \cdot \zeta (G_2)\,,
\end{eqnarray*}
where $|G|$ and $\|G\|$ denote the number of vertices and the number of edges of $G$ respectively.

Other graph products, such as disjunction, symmetric difference, composition, sum, join, corona
product were examined in \cite{DoSa09}. The results can be applied to some graphs of chemical
interest.

For the rectangular grid $P_a \cp P_b$\,, it follows
$$
\xi(P_a \cp P_b) = \left\{
\begin{array}{l}
3a^2b + 3ab^2 - 8ab + 3a + 3b - \frac{3a^2}{2} - \frac{3b^2}{2}, \phantom{+0}\quad \qquad \ \ \mbox{if $a$ and $b$ are even};\\
3a^2b + 3ab^2 - 8ab + 3a + 2b - \frac{3a^2}{2} - \frac{3b^2}{2} + \frac{1}{2},\quad \qquad \mbox{if $a$ is odd and $b$ is even};\\
3a^2b + 3ab^2 - 8ab + 2a + 3b - \frac{3a^2}{2} - \frac{3b^2}{2} + \frac{1}{2},\quad \qquad \mbox{if $a$ is even and $b$ is odd};\\
3a^2b + 3ab^2 - 8ab + 2a + 2b - \frac{3a^2}{2} - \frac{3b^2}{2} + 1,\quad \qquad \mbox{if $a$ and
$b$ are odd}.
\end{array}
\right.
$$

A $C_4$ nanotorus is a Cartesian product of two cycles, $C_a \cp C_b$\,,
$$
\xi^c (C_a \cp C_b) = 4ab \left ( \left \lfloor \frac{a}{2} \right \rfloor + \left \lfloor
\frac{b}{2} \right \rfloor \right),
$$
while $C_4$ nanotube is a Cartesian product of path and cycle, $P_a \cp C_b$\,,
$$
\xi(P_a \cp C_b) = \left\{
\begin{array}{l}
3a^2b + 2ab^2 - 4ab - b^2 + 2b,\quad \qquad \mbox{if $a$ and $b$ are even};\\
3a^2b + 2ab^2 - 6ab - b^2 + 3b,\quad \qquad \mbox{if $a$ is even and $b$ is odd};\\
3a^2b + 2ab^2 - 4ab - b^2 + \phantom{3}b,\quad \qquad \mbox{if $a$ is odd and $b$ is even};\\
3a^2b + 2ab^2 - 6ab - b^2 + 2b,\quad \qquad \mbox{if $a$ and $b$ are odd}.
\end{array}
\right.
$$

\section{Linear algorithm for trees}

Here we present a linear algorithm for calculating the eccentric connectivity index of weighted
trees. In \cite{Il10}, the author presented an extension of this algorithm for unicyclic graphs.

Let $T$ be a rooted tree. By $T_r$ denote the subtree rooted at vertex $r \in V$\,, which is a
subgraph induced on vertex $r$ and all of its descendants.

The adjacency matrix $adj$ is used to store the weights of the edges. First, we choose an arbitrary
vertex as the root and start depth first search procedure (see \cite{CoLRS01} for details). For
every vertex $v$\,, we have to find the length of the longest path from $v$ in the subtree rooted
at $v$\,. In each recursive call we maintain two arrays:
\begin{itemize}
\item $ddown [v]$ representing the length of the longest path from the vertex $v$ in the subtree $T_v$\,;
\item $parent [v]$ representing the unique parent of $v$ in the DFS tree (for the root vertex, we have
$parent [v] = -1$).
\end{itemize}

Next, we perform the second function \verb"Eccentricity(root)", Algorithm \ref{alg2}. In the array
$dup [v]$ we maintain the length of the longest upwards path from the vertex $v$ such that the
first vertex on the path is the parent of $v$\,. We can calculate the value $dup [v]$ in the
recursive procedure by examining all children of the vertex $parent [v]$\,. Finally, for every
vertex $v$ we sum $deg [v] \cdot ecc [v]$\,, where $ecc [v] = \max (ddown [v], dup [v])$ and get
$\xi^c (T)$\,.

\begin{procedure}
\label{alg1} 
    \KwIn{The adjacency list of the tree $T$ with the root vertex $root$\,.}
    \KwOut{The arrays $ddown$ and $parent$\,.}
    \BlankLine

    $ddown [v] = 0$\;
    \ForEach{neighbor $u$ of $v$}
    {
        \If{$(parent [u] = -1)$ and $(u \neq root)$}
        {
            $parent [u] = v$\;
            $DFS (u)$\;
            \If{$ddown [v] < adj [v, u] + ddown [u]$}
            {
                $ddown [v] = adj [v, u] + ddown [u]$\;
            }
        }
    }
    \caption{DFS (vertex v)}
\end{procedure}

\begin{procedure}
\label{alg2} 
    \KwIn{The adjacency list of the tree $T$ and the arrays $ddown$ and $parent$\,.}
    \KwOut{The eccentricity of every vertex stored in the array $ecc$\,.}
    \BlankLine


    $dup [v] = 0$\;
    \If{$parent [v] \neq -1$}
    {
        \ForEach{neighbor $u$ of $parent [v]$}
        {
            \If{$(u \neq v)$ \mbox{\bf and} $(u \neq parent [parent [v]])$}
            {
                $dup [v] = \max (dup [v], adj [parent [v], u] + ddown [u])$\;
            }
        }
        $dup [v] = adj [v, parent [v]] + \max (dup [v], dup [parent [v]])$\;
    }

    \ForEach{neighbor $u$ of $v$}
    {
        \If{$parent [v] \neq u$}
        {
            $Eccentricity (u)$\;
        }
    }
    $ecc [v] = \max (ddown [v], dup [v])$\;

    \caption{Eccentricity (vertex v)}
\end{procedure}

The time complexity of the algorithm is linear $O (n)$\,, and the memory used is $O (n)$\,, since
we need three additional arrays of length $n$ while the adjacency list contains $2 (n - 1)$
elements.

\section{Concluding remarks}

In this paper we survey the mathematical properties of eccentric connectivity index. We present the
extremal trees and unicyclic graphs with maximum and minimum eccentric connectivity index subject
to certain constraints (diameter, radius, the number of pendent vertices, matching number,
independence number, maximum vertex degree, girth). Sharp lower and asymptotic upper bounds are
given and various connections with other important graph invariants are established (Zagreb index,
Wiener index, degree distance). We present explicit formulae for the eccentric connectivity index
for several families of composite graphs and designed a linear algorithm for calculating the
eccentric connectivity index of trees.

There are still many interesting open questions for the further study. It would be interesting to
compute the values of $\xi^c (G)$ for various classes of dendrimers or general linear polymers, and
to introduce the ordering of trees and unicyclic graphs with respect to $\xi^c$\,. It would be also
interesting to determine extremal regular (cubic) graphs with respect to the eccentric connectivity
index. In the following, we present some derivative indices of the eccentric connectivity index and
research of those indices seems the most natural course of the future work.

The {\em eccentric distance sum} is defined in \cite{GuSiMa02} as
$$
\xi^{ds}(G) = \sum_{v \in V (G)} \varepsilon (v) \cdot D (v)\,,
$$
where $D (v) = \sum_{u \in V (G)} d (u, v)$ is the sum of all distances from the vertex $v$\,. This
index offers a vast potential for structure activity/property relationships.

The {\em adjacent eccentric distance sum index} is defined in \cite{SaMa03} as
$$
\xi^{sv}(G) = \sum_{v \in V (G)} \frac{\varepsilon (v) \cdot D (v)}{deg (v)}\,.
$$

The {\em augmented and super augmented eccentric connectivity indices}
\cite{BaSaGuMa06,BaSaMa05,DuMa09,DuGuMa08} are novel modifications of the eccentric connectivity
index with augmented discriminating power,
$$
\xi^{ac}(G) = \sum_{v \in V (G)} \frac{M (v)}{\varepsilon (v)} \qquad \mbox{and} \qquad
\xi^{sac}(G) = \sum_{v \in V (G)} \frac{M (v)}{\varepsilon^2 (v)}\,,
$$
where $M (v)$ is the product of degrees of all neighboring vertices of $v$\,. These indices were
found to exhibit high sensitivity towards the presence and relative position of heteroatoms.

\vspace{0.5cm}

{\bf Acknowledgement. } This work was supported by the Research grant 144007 of the Serbian
Ministry of Science and Technological Development. I would like to thank Professors Ivan Gutman and
Mircea Diudea for their continuous advices and encouragement. Also, I would like to thank Andreja
Ili\' c and Tomislav Do\v sli\' c for valuable discussions and providing some recent references.


\end{document}